\documentclass[leqno,11pt]{article}
\usepackage{amsmath, amssymb}
\usepackage{amsthm}

\numberwithin{equation}{section}
\theoremstyle{plain} %
\newtheorem{theorem}{\noindent\sc \bf{Theorem}}[section] %
\newtheorem{lemma}[theorem]{\noindent\sc \bf{Lemma}}

\newtheorem{proposition}[theorem]{\noindent\sc \bf{Proposition}}

\newtheorem{thm}{\noindent\sc \bf{Theorem}} %

\newtheorem{cor}[thm]{\noindent\sc \bf{Corollary}}
\theoremstyle{definition} %

\newcommand{\fulltoday}{\number\day\space \ifcase\month\or
    January\or February\or March\or April\or May\or June\or
    July\or August\or September\or October\or November\or December\fi
    \space\number\year}

\newcommand{\bM}{{\bf M}}

\newcommand{\C}{\mathbb{C}}
\newcommand{\NN}{\mathbb{N}}
\newcommand{\R}{\mathbb{R}}
\newcommand{\Q}{\mathbb{Q}}
\newcommand{\Z}{\mathbb{Z}}

\newcommand{\SL}{\mathrm{SL}}
\newcommand{\GL}{\mathrm{GL}}
\newcommand{\M}{\mathrm{M}}
\newcommand{\Mp}{\mathrm{Mp}}
\newcommand{\Gr}{\mathrm{Gr}}
\newcommand{\0}{\mathbf{0}}

\newcommand{\gre}{\mathfrak{e}}

\newcommand{\grH}{\mathfrak{H}}

\newcommand{\e}{\mathbf{e}}

\newcommand{\tp}{{}^t}
\renewcommand{\Re}{{\rm Re}}
\renewcommand{\Im}{{\rm Im}}
\newcommand{\cC}{\mathcal{C}}
\newcommand{\cD}{\mathcal{D}}

\newcommand{\cF}{\mathcal{F}}

\newcommand{\cH}{\mathcal{H}}

\newcommand{\cM}{\mathcal{M}}

\newcommand{\cT}{\mathcal{T}}

\newcommand{\cV}{\mathcal{V}}

\newcommand{\cX}{\mathcal{X}}
\newcommand{\cY}{\mathcal{Y}}

\newcommand{\inv}{^{-1}}

\newcommand{\Id}{{\rm Id}}

\newcommand{\cross}{^{\times}}
\newcommand{\x}{^{\times}}
\newcommand{\up}{^{\uparrow}}
\newcommand{\down}{^{\downarrow}}

\newcommand{\half}{\dfrac{1}{2}}
\newcommand{\vecm}[1]{\left( \begin{array}{c} #1 \end{array} \right)}
\newcommand{\Nmat}[2]{\left( \begin{array}{cc} #1 \\ #2 \end{array} \right)}
\newcommand{\Smat}[3]{\left( \begin{array}{ccc} #1 \\ #2 \\ #3 \end{array} \right)}

\renewcommand{\pmod}[1]{\; (\bmod \; #1)}
\newcommand{\form}[2]{\langle #1 ,#2  \rangle}

\newcommand{\bs}{\backslash}

\newcommand{\oo}{_{\infty}}

\newcommand{\ol}[1]{\overline{#1}}

\newcommand{\act}[1]{\langle #1 \rangle}

\begin{document}

\begin{center}
 {\Large \textbf{Symmetries for 
Siegel Theta Functions, \\
 Borcherds Lifts and Automorphic Green Functions\\[0.4cm]}}
{\large \textbf{Bernhard Heim and Atsushi Murase}}\\[0.4cm]
\end{center}


\noindent
\textbf{Abstract.}\ 
Let $q$ be an integral  quadratic form 
of signature $(2,m+2)$.
We will show that the Siegel theta functions attached to $q$ 
satisfies certain symmetries.
As an application, we prove the symmetries for automorphic forms
on the orthogonal group of $q$ closely related to Heegener divisors
 (Borcherds lifts and automorphic Green functions).

\noindent
%

\section{Introduction}
\subsection{The main results}
Let $(L,q)$ be a quadratic space over $\Z$ of signature
$(b^+,b^-)$ with $b^-\ge b^+>0$, and $G$
  the orthogonal group of $q$.
Let $\Gr(L)$ be
 the set of $b^+$-dimensional subspaces $v$ of
$V=L\otimes_{\Z}\R$ such that $q|_v>0$.
The Grassmannian $\Gr(L)$ of $L$ is a real analytic manifold on which
$G(\R)$ acts  in a natural manner.
For $\lambda\in V$ and $v\in\Gr(L)$, we denote by 
$\lambda_v$ and $\lambda_v^{\perp}$  the projections of 
$\lambda$ to $v$ and $v^{\perp}$ respectively, where
$v^{\perp}$ is the orthogonal complement  of $v$
in $V$ with respect to $q$.
Let $L^*$ be the dual lattice of $L$ and let
$\grH=\{\tau\in\C\mid\Im(\tau)>0\}$ be
the upper half plane.
For $\alpha\in L^*/L$ and $(\tau,v)\in \grH\times \Gr(L)$, we define
the \emph{Siegel theta function}
by
\[
 \Theta_{\alpha}(\tau,v)=\sum_{\lambda\in \alpha+L}\exp
\left(2\pi i \left(\tau q(\lambda_v)+\ol{\tau}q(\lambda_v^{\perp})\right)\right).
\]
The Siegel theta function, first introduced by
Siegel (\cite{Si}), is of fundamental importance in number theory.
In particular it plays an crucial role 
in the arithmetic of quadratic forms and  the theory of
automorphic forms of several variables.

From now on we
assume that $b^+=2$, in which case $\Gr(L)$ has a structure of hermitian
symmetric domain of type IV.
The object of the paper is to show that certain 
symmetries hold for the Siegel
theta functions. As an application of this result,
we obtain similar symmetries for  automorphic forms on 
$\Gr(L)$ closely related to Heegner divisors (Borcherds lifts and automorphic
Green functions).

To state our main results more precisely, let $m$ be a nonnegative
integer and
$S$  an even integral positive definite symmetric matrix of degree $m$.
Let $L=\Z^{m+4}$ and $q$ a quadratic form given by 
\[
 q(x)=\half\,  \tp x Qx \qquad(x\in L),
\]
where
\[
 Q=
\left(
\begin{matrix}
 &&&&1 \\
&&&1&\\
&&-S&&\\
&1&&&\\
1&&&&
\end{matrix}
\right).
\]
Note that we include the case of $m=0$, in which case
$q(x)=x_1x_4+x_2x_3$ for $x=\tp(x_1,x_2,x_3,x_4)\in
\Z^4$.

It is known that $\Gr(L)$ is isomorphic to the tube domain $\cD=\{(z,w,z')\in
\grH\times \C^m\times \grH\mid 2\,\Im(z)\Im(z')-\tp \Im(w)\,S\,\Im(w)>0\}$.
We henceforth consider $\Theta_{\alpha}$ as a function on $\grH\times \cD$.
The main result of the paper
 is  the following \emph{additive symmetry} for  Siegel theta
functions.
\begin{thm}
\label{th:A}
 For $\alpha\in L^*/L$ and a natural number $n$, we have
\[
 \sum_{a,b,d}\Theta_{\alpha}\left(\tau,\left(\dfrac{az+b}{d},
\dfrac{\sqrt{n}}{d}w,
z'\right)\right)
=\sum_{a,b,d}\Theta_{\alpha}\left(\tau,\left(z,\dfrac{\sqrt{n}}{d}w,
\dfrac{az'+b}{d}\right)\right),
\]
where $(a,b,d)$ runs over the set 
$\cH_n=\{(a,b,d)\in\Z^3 \mid a,d>0,ad=n,0\le b<d\}$.
\end{thm}

 The notion of the additive symmetry was first introduced 
by the first named author
 in the Siegel modular case (\cite{H}).
In \cite{HeMu}, we defined the additive symmetry for automorphic forms 
on $O(2,m+2)$ and  showed  that the  Maass space is
 characterized
by the additive symmetry.

To state applications of Theorem \ref{th:A}, we recall the definition of
Heegner divisors.
For $\lambda\in V$, let $D(\lambda)$ be the divisor on $\cD$
corresponding to $\{v\in \Gr(L)\mid
Q(\lambda,v)=0\}$.
If $q(\lambda)<0$, $D(\lambda)$ is isomorphic to a hermitian
symmetric domain of type IV associated with $q|_{\lambda^{\perp}}$,
where $\lambda^{\perp}$ denotes the orthogonal complement of $\lambda$
in $V$ with respect to $q$.
For $\alpha\in L^*/L$ and $n\in q(\alpha)+\Z$ with $n<0$,
let
\[
 H(\alpha,n)=\sum_{\lambda\in \alpha+L,\,q(\lambda)=n}D(\lambda)
\]
be a divisor on $\cD$.
Then $H(\alpha,n)$ is invariant under the action of $\Gamma^*(L)=
\{\gamma\in G(\R)^+\mid \gamma L=L,\gamma|_{L^*/L}=\Id\}$,
and defines an algebraic divisor on $\Gamma^*(L)\bs\cD$.
Here  $G(\R)^+$ denotes the identity component of $G(\R)$.
The Heegner divisor $H(\alpha,n)$ plays an important role in the
arithmetic
of Shimura varieties attached to $G$ (for example see \cite{HZ},
\cite{GZ}, \cite{vdG}).

Borcherds (\cite{Bo2}) constructed a meromorphic automorphic form
on $\Gamma^*(L)$ whose divisor is a linear combination of Heegner
divisors with coefficients in $\Z$. Such an automorphic form is called
a \emph{Borcherds lift}.
A Borcherds lift $\Psi$ is obtained as, essentially,
the exponential of a \emph{regularized} integral of a weakly
holomorphic modular form (for a precise definition, see \textbf{5.3})
against the Siegel theta function.
Thus Theorem \ref{th:A} implies 
the following \emph{multiplicative symmetry} for  Borcherds lifts:

\begin{cor}
\label{cor:B}
 Let $\Psi$ be a  Borcherds lift. Then we have
\[
 \prod_{a,b,d}\Psi\left(\dfrac{az+b}{d},\dfrac{\sqrt{n}}{d}w,
z'\right)
=\epsilon_n(\Psi)
\,
\prod_{a,b,d}\Psi\left(z,\dfrac{\sqrt{n}}{d}w,\dfrac{az'+b}{d}\right)
\]
for any natural number $n$.
Here $(a,b,d)$ runs over 
$\cH_n$ and
 $\epsilon_n(\Psi)$ is a complex number of absolute value $1$
depending on $n$ and $\Psi$.
\end{cor}

Finally let $\Phi_{\alpha,n}$ denote the automorphic Green function 
associated with the Heegner divisor $H(\alpha,n)$, which has been
introduced and studied by Bruinier
(\cite{Br1}, \cite{Br2}) and Oda and Tsuzuki (\cite{OT}) independently
(for a precise
definition of $\Phi_{\alpha,n}$,
see \textbf{6.1}).
Bruinier (\cite{Br2}) showed that $\Phi_{\alpha,n}$ is obtained as 
a regularized integral of a certain Poincar\'{e} series against 
the Siegel theta function.
Thus, again, Theorem \ref{th:A} implies

\begin{cor}
 \label{cor:C}
The automorphic Green function
$\Phi_{\alpha,n}$ satisfies the additive symmetry.
\end{cor}

\noindent
\textbf{Remark.}\quad 
The above results (Corollaries \ref{cor:B} and \ref{cor:C})
suggest that certain symmetries should hold for Heegner divisors.
We hope to study these symmetries in future. 

\bigskip

\noindent
\textbf{Remark.}\quad 
 In this paper, we 
assume that the Witt index of $q$ is equal to $2$ and hence
 the $\Q$-rank of $G$ is 
equal to $2$.
In the forthcoming paper, we will treat the Hilbert modular case 
($G=O(2,2)$ of $\Q$-rank $1$), which contains some special features.

\subsection{The organization of the paper}
The paper is organized as follows.
In Section 2, after recalling the definition of automorphic forms
on $O(2,m+2)$,
we introduce the notions of the additive symmetry  and the 
multiplicative symmetry.
In Section 3, we recall the definition of Siegel theta functions and
state  the main result of the paper (Theorem \ref{th:main-1}),
the additive symmetry for  Siegel theta functions.
The definition of Heegner divisors is recalled in Section 4.
In Section 5 and 6, we prove the multiplicative (respectively additive)
symmetry for Borcherds lifts (respectively automorphic Green functions)
assuming Theorem \ref{th:main-1}.
Section 7  is devoted to the proof of Theorem \ref{th:main-1},
which  relies on a formula due to Borcherds (Proposition \ref{prop:BF})
and  Poisson summation formula.
In the final section, we study several examples of Siegel
modular forms of degree 2. In particular, as an application of the
multiplicative symmetry for Borcherds lifts,
we show that a holomorphic
Siegel cusp form of weight $12$ on $\mathrm{Sp}_2(\Z)$, which
is unique up to constant multiples, is \emph{not} a 
Borcherds lift.

\subsection{Notation}
As usual, we denote by $\NN, \Z, \Q, \R$ and $\C$
the set of natural numbers, the ring of rational integers,
the field of  rational numbers, the field of  real numbers and
the field of  complex numbers respectively.
For a real symmetric matrix $T$ of degree $n$, let
$T(x,y)=\tp xTy$ and $T[x]=T(x,x)$ for $x,y\in \C^n$.
We put $\e(z)=\exp(2\pi i z)$ for $z\in \C$.
Denote by $\delta_{ij}$  the Kronecker's delta.
For a condition $C$, we put 
\[
 \delta(C)=
\begin{cases}
 1 & \text{if $C$ holds},\\
 0 & \text{otherwise}.
\end{cases}
\]

\section{Automorphic forms on $O(2,m+2)$ and the symmetry}
Throughout the paper, we fix a positive definite even integral 
symmetric matrix $S$ of degree $m\ge 0$. For $\nu\in \Z_{\ge 0}$, put
\[
 Q_{\nu}=\Smat{0&0&J_{\nu}}{0&-S&0}{J_{\nu}&0&0},
\]
where $J_{\nu}=\left(\delta_{i,\nu-j+1}\right)_{1\le i,j\le \nu}\in \GL_{\nu}$.
Then $Q_{\nu}$ is an even integral symmetric matrix of
signature  $(\nu,m+\nu)$.
Let
$L_{\nu}=\Z^{m+2\nu}, L_{\nu}^*=Q_{\nu}\inv L_{\nu},
V_{\nu}=L_{\nu}\otimes_{\Z}\R=\R^{m+2\nu}$ and
$V_{\nu,\C}=V_{\nu}\otimes_{\R}
\C=\C^{m+2\nu}$.
We often write $Q, L, L^*$ and $V$ for $Q_{2},L_{2},L^*_{2}$ and $V_{2}$
respectively.
Thus
\[
 Q=
\left(
\begin{matrix}
 &&&&1 \\
&&&1&\\
&&-S&&\\
&1&&&\\
1&&&&
\end{matrix}
\right),\ L=\Z^{m+4}, \,L^*=Q\inv L,\,V=\R^{m+4}.
\]
We often write $(x_1,x_2,x_3)$ for 
\[
 \vecm{x_1\\x_2\\x_3}\in V_{1,\C}\qquad
(x_1,x_3\in\C,x_2\in V_{0,\C}=\C^m)
\]
if there is no fear of confusion.
Let $q$ be a quadratic form on $V$ defined by
$q(X)=2\inv Q[X]$
for $X\in V$.

Denote by $G=O(Q)$  the orthogonal group of $Q$.
Let $G(\R)^+$ be the  identity component of $G(\R)$
 and
\[
 \cD=\left\{Z=
(z,w,z')\in V_{1,\C}
\mid z,z'\in \grH,w\in \C^m, Q_1[\Im(Z)]=2\,\Im(z)\Im(z')-S[\Im(w)]>0
\right\}.
\]
As is well-known, $\cD$ is a hermitian symmetric domain of type IV.
Define an action of $G(\R)^+$ on $\cD$ and an automorphic factor
$J\colon G(\R)^+ \times \cD\to \C\x$ as follows:
For $Z\in \cD$, put
\[
 \widetilde Z=\vecm{-Q_1[Z]/2 \\ Z \\ 1}\in V_{\C}.
\]
Note that $Q[\widetilde Z]=0$ and $Q(\widetilde Z,\ol{\widetilde Z})=2
Q_1[\Im(Z)]>0$.
For $g\in G(\R)^+$ and $Z\in \cD$, we define $g\act{Z}\in \cD$ and
$J(g,Z)\in \C\cross$ by
\[
 g \widetilde Z=\widetilde{g\act{Z}}\, J(g,Z).
\]
Let $k$ be an integer  and $F$  a function on $\cD$.
For $g\in G(\R)^+$, we define the Petersson slash operator by
$(F|_k g)(Z)=J(g,Z)^{-k}F(g\act{Z})$.

Let
$\Gamma(L)=\{\gamma\in G(\R)^+\mid \gamma L=L\}$.
For an index finite subgroup $\Gamma$ of $\Gamma(L)$,
a unitary character $\chi$ of $\Gamma$ and $k\in\Z$, 
denote by $\cC_k(\Gamma,\chi)$  the space of
smooth functions $F$ on $\cD$ satisfying $F|_k\gamma=\chi(\gamma)F$
for any $\gamma\in\Gamma$.

Define two embeddings of $\SL_2$ into $G$ by
\begin{align*}
 \iota\up(h)&=\left(
\begin{matrix}
 a&  &  &-b& \\
  & a&  &  & b\\
  &  &1_m& & \\
-c&  &   & d&\\
  &c &   &  &d
\end{matrix}
\right),\quad
\iota\down(h)=
\left(
\begin{matrix}
 a& -b &  && \\
-c  & d&  &  & \\
  &  &1_m& & \\
&  &   & a&b\\
  & &   & c &d
\end{matrix}
\right)
\end{align*}
for $h=\Nmat{a&b}{c&d}\in\SL_2$.
These embeddings commute with each other, and $\iota\up(\SL_2(\R))$ and
$\iota\down(\SL_2(\R))$ are contained in $G(\R)^+$.

For $F\in \cC_k(\Gamma,\chi)$ and  $n\in \NN$, we 
define
\begin{align*}
 F|_k T\up_n(Z)&=n^{k/2-1}
\sum_{a,b,d} F|_k\iota\up\left(\sqrt{n}\inv \Nmat{a&b}{0&d}\right)(Z)\\
&=\dfrac{1}{n}\,
\sum_{a,b,d}\left(\dfrac{n}{d}\right)^k
F\left(\dfrac{az+b}{d},
\dfrac{\sqrt{n}}{d}w,
z'\right),\\
 F|_k T\down_n(Z)&=n^{k/2-1}
\sum_{a,b,d} F|_k\iota\down\left(\sqrt{n}\inv \Nmat{a&b}{0&d}\right)(Z)\\
&=\dfrac{1}{n}\,
\sum_{a,b,d}\left(\dfrac{n}{d}\right)^k
F\left(\tau,\left(z,
\dfrac{\sqrt{n}}{d}w,
\dfrac{az'+b}{d}\right)\right),
\end{align*}
where $Z=(z,w,z')\in\cD$ and
 $(a,b,d)$ runs over 
\[
 \cH_n=\{(a,b,d)\in\Z^3 \mid a,d>0,ad=n,0\le b<d\}.
\]
Note that $F|_kT\up_n$ and $F|_kT\down_n$ are not  in
$\cC_k(\Gamma,\chi)$ in general.
For a prime $p$, we have
\[
 F|_k T\up_p(z,w,z')=p^{k-1}F\left(pz,\sqrt{p}w,z' \right)+
  p^{-1}\sum_{a=0}^{p-1}F\left(p\inv(z+a),\sqrt{p}\inv w,z'\right)
\]
and
\[
 F|_k T\down_p\left(z,w,z'\right)=
p^{k-1}F\left(z, \sqrt{p}w,pz'\right)+
  p^{-1}\sum_{a=0}^{p-1}F\left(z, \sqrt{p}\inv w,p\inv(z'+a)\right).
\]

We say that $F\in \cC_k(\Gamma,\chi)$ satisfies the \emph{additive symmetry}
if the equality
\begin{equation}
\label{eq:as}
  F|_kT\up_n=F|_kT\down_n
\end{equation}
holds for any $n\ge 1$.
It is easily seen that 
$F$ satisfies the additive symmetry
if and only if $F$ satisfies (\ref{eq:as}) for $n=p$, $p$ any prime number.
 
We also define
\begin{align*}
 F|\cT\up_n(Z)&=
\prod_{a,b,d} 
F\left(\iota\up\left(\sqrt{n}\inv \Nmat{a&b}{0&d}\right)\act{Z}\right)\\
&=
\prod_{a,b,d}
F\left(\dfrac{az+b}{d},
\dfrac{\sqrt{n}}{d}w,
z'\right),\\
 F|\cT\down_n(Z)&=
\prod_{a,b,d}
F\left(\iota\down\left(\sqrt{n}\inv \Nmat{a&b}{0&d}\right)\act{Z}\right)\\
&=
\prod_{a,b,d}
F\left(z,
\dfrac{\sqrt{n}}{d}w,
\dfrac{az'+b}{d}\right),
\end{align*}
where $Z=(z,w,z')\in\cD$ and $(a,b,d)$ runs over $\cH_n$.
We say that $F$ satisfies the \emph{multiplicative symmetry}
if 
\begin{equation}
\label{eq:ms}
  F|\cT_n\up=\epsilon_{n}(F)\, F|\cT_n\down
\end{equation}
holds for any $n\ge 1$ with  a complex number  $\epsilon_{n}(F)$
of absolute value $1$ depending on $n$ and $F$.
Similarly as above, 
$F$ satisfies the multiplicative symmetry
if and only if $F$ satisfies (\ref{eq:ms}) for $n=p$, $p$ any prime number.


\section{Siegel theta functions}

\subsection{The Grassmannian}

Recall that the Grassmannian $\Gr(L)$ of $L$ is
 the set 
 consisting of $2$-dimensional subspaces of $V$ on
which
$q$ is positive definite.
Then $\Gr(L)$ is  a hermitian symmetric domain of type IV
on which $G(\R)^+$ acts  transitively in a natural manner,
and 
isomorphic to $\cD$.
The isomorphism from $\cD$ to $\Gr(L)$ is given by
$Z\in\cD\mapsto v_Z$,
the subspace of $V$ generated by $\Re(\widetilde Z)$
and $\Im(\widetilde Z)$.

\subsection{Siegel theta functions}

For $\lambda\in V$ and $v\in \Gr(L)$, let $\lambda_v$ and $\lambda_{v}^{\perp}$
be the projections of $\lambda$ to $v$ and $v^{\perp}$ respectively,
where $v^{\perp}$ denotes the orthogonal complement
of  $v$ in $V$
with respect to $q$.
Then $q(\lambda)=q(\lambda_v)+q(\lambda_{v}^{\perp})$. 
For $\alpha\in L^*/L$, 
the Siegel theta function $\Theta_{\alpha}$ is defined by
\begin{align}
 \Theta_{\alpha}(\tau,v)
&=\sum_{\lambda\in \alpha+L}\,\e
\left(\tau q(\lambda_v)+\ol{\tau}q(\lambda_{v}^{\perp})\right)\\
&\notag =\sum_{\lambda\in \alpha+L}\,\e
\left(i\,\Im(\tau)Q[\lambda_v]+\dfrac{\ol{\tau}}{2}\,Q[\lambda]\right)
\end{align}
for $\tau\in \grH$ and $v\in \Gr(L)$.
As a function of $v$, $\Theta_{\alpha}(\tau,v)$ is invariant under 
\begin{equation}
 \Gamma^*(L)=\{\gamma\in \Gamma(L)\mid \gamma \lambda\equiv
\lambda\pmod{L}
\text{ for any }\lambda\in L^*
\},
\end{equation}
the discriminant group of $L$.
For the automorphy of $\Theta_{\alpha}(\tau,v)$ with respect to $\tau$,
see \textbf{5.2}.
By abuse of notation, we often write $\Theta_{\alpha}(\tau,Z)$ for 
$\Theta_{\alpha}(\tau,v_Z)$.
Note that $\Theta_{\alpha}(\tau,Z)\in \cC_0(\Gamma^*(L),\mathbf{1})$ as a
function of $Z\in \cD$,
where $\mathbf{1}$ stands for the trivial character of $\Gamma^*(L)$.
We easily see that
\[
 \Theta_{\alpha}(\tau,Z)=\sum_{\lambda\in \alpha+L}
\e\left(i\,\Im(\tau)\dfrac{|Q(\lambda,\widetilde Z)|^2}{Q_1[\Im(Z)]}+
\dfrac{\ol{\tau}}{2}\,Q[\lambda]\right).
\]
The main result of the paper is stated as follows.

\begin{theorem}
\label{th:main-1}
 Let $\alpha\in L^*/L$ and $\tau\in \grH$.
Then $Z\mapsto \Theta_{\alpha}(\tau,Z)$ satisfies the additive
 symmetry. 
Namely, for any $n\in\NN$, we have
\[
 \sum_{a,b,d}\Theta_{\alpha}\left(\tau,\left(\dfrac{az+b}{d},
\dfrac{\sqrt{n}}{d}w,
z'\right)\right)
=\sum_{a,b,d}\Theta_{\alpha}\left(\tau,\left(z,\dfrac{\sqrt{n}}{d}w,
\dfrac{az'+b}{d}\right)\right),
\]
where $(a,b,d)$ runs over $\cH_n$.
\end{theorem}

We postpone the proof of the theorem until Section 7.

\section{Heegner divisors}

The quotient $\cX_L=\Gamma^*(L)\bs \cD$
is a quasi-projective algebraic variety over $\C$ of dimension $m+2$.
For $\lambda\in V$ with $q(\lambda)<0$, let
\[
 D(\lambda)=\{Z\in \cD\mid Q(\lambda,\widetilde Z)=0\}
\]
be a complex analytic divisor on $\cD$.
For $\alpha\in L^*/L$ and $n\in q(\alpha)+\Z$ with $n<0$,
define
\[
 H(\alpha,n)=\sum_{\lambda\in L+\alpha,\,q(\alpha)=n}
D(\lambda).
\]
Then $H(\alpha,n)$ is a $\Gamma^*(L)$-invariant divisor on $\cD$,
called the \emph{Heegner divisor} of discriminant $(\alpha,n)$.
It is known that $H(\alpha,n)$ is the inverse image under the canonical
projection of an algebraic divisor on $\cX_L$, also denoted by
$H(\alpha,n)$ (for example see \cite{Br2}, \textbf{2.2}).


\section{Borcherds lifts}

In this section, we recall the definition of Borcherds lifts after \cite{Bo2}
and \cite{Br2},
and prove the multiplicative symmetry for Borcherds lifts
assuming Theorem \ref{th:main-1}.

\subsection{Metaplectic representations}
Let $\Mp_2(\R)$ be the metaplectic group. By definition,
$\Mp_2(\R)$ consists of $(M,\phi)$, where $M\in \SL_2(\R)$ and 
$\phi$ is a holomorphic function on $\grH$ satisfying
$\phi(\tau)^2=j(M,\tau)$,
and the product is given by
$(M_1,\phi_1(\tau))(M_2,\phi_2(\tau))
=(M_1M_2,\phi_1(M_2\act{\tau})\,\phi_2(\tau))$.
Here
\[
 M\act{\tau}=\dfrac{a\tau+b}{c\tau+d}\,,\ j(M,\tau)=c\tau+d\qquad
\left(M=\Nmat{a&b}{c&d}\in\SL_2(\R),\tau\in\grH\right)
\]
as usual.
Let $\Mp_2(\Z)$ be the inverse image of $\SL_2(\Z)$ under the natural
projection $\Mp_2(\R)\to \SL_2(\R)$.
It is known that $\Mp_2(\Z)$ is generated by
\[
 T=\left(\Nmat{1&1}{0&1},1\right),\ 
 S=\left(\Nmat{0&-1}{1&0},\sqrt{\tau}\right).
\]

Let $\{\gre_{\alpha}\}_{\alpha\in L^*/L}$ be the standard basis of the 
group ring $\C[L^*/L]$ with
$\gre_{\alpha}\gre_{\alpha'}=\gre_{\alpha+\alpha'}$.
Let $\form{}{}$ denote the inner product on $\C[L^*/L]$
defined by
$\form{\sum_{\alpha}x_{\alpha}\gre_{\alpha}}{\sum_{\alpha}y_{\alpha}\gre_{\alpha}}=\sum_{\alpha}x_{\alpha}\ol{y_{\alpha}}$,
where $\alpha$ runs over $L^*/L$.
There exists a unitary representation $\rho_L$ of $\Mp_2(\Z)$
on $\C[L^*/L]$ defined by
\begin{align*}
 \rho_L(T)\gre_{\alpha}&=\e(q(\alpha))\gre_{\alpha},\\
 \rho_L(S)\gre_{\alpha}&=\dfrac{\exp\left(\dfrac{\pi i}{4}m\right)}
 {\sqrt{|L^*/L|}}
                      \sum_{\beta\in L^*/L}\e(-Q(\alpha,\beta))\gre_{\beta}
\end{align*}
(see \cite{Bo2}, \S 2 and \cite{Br2}, \textbf{1.1}).

For a function $f$ on $\grH$ with values in $\C[L^*/L]$, $k\in 2\inv \Z$ and 
$(M,\phi)\in\Mp_2(\Z)$, we put
\[
 f|_k(M,\phi)(\tau)=\phi(\tau)^{-2k}\rho_L(M,\phi)\inv f(M\act{\tau}).
\]

\subsection{Automorphy of Siegel theta functions}
Set
\begin{equation}
 \label{eq:Theta}
\Theta_L(\tau,Z)=\sum_{\alpha\in L^*/L}\gre_{\alpha}\Theta_{\alpha}(\tau,Z).
\end{equation}
For $(M,\phi)\in \Mp_2(\Z)$, we have
\begin{equation}
 \label{eq:Theta-transformation}
 \Theta_L(M\act{\tau},Z)
=\phi(\tau)^2\,\ol{\phi(\tau)}^{m+2}\rho_L(M,\phi)\Theta_L(\tau,Z)
\end{equation}
(see \cite{Bo2}, Theorem 4.1).

\subsection{Weakly holomorphic modular forms}
For $k\in 2\inv \Z$, let $\cM_k(\rho_L)$ be 
 the space of
holomorphic
functions $f$ on $\grH$ with values in $\C[L^*/L]$ satisfying
the following two conditions:
\begin{enumerate}
 \item For $(M,\phi) \in \Mp_2(\Z)$, we
       have $f|_k(M,\phi)=f$.
 \item Let
\[
 f(\tau)=\sum_{\alpha\in L^*/L,\, n\in
       q(\alpha)+\Z}
        \,c(\alpha,n)
       \gre_{\alpha}(n\tau)
\]
be the   Fourier expansion of $f$, where we put
$\gre_{\alpha}(z)=\e(z)\,\gre_{\alpha}$ for $z\in\C$.
 For every
$\alpha\in L^*/L$, we have
 $c(\alpha,n)=0$ if $n\ll 0$.
\end{enumerate}
We call $\cM_k(\rho_L)$
the space of
\emph{weakly holomorphic modular forms}
of weight $k$ with respect to $\rho_L$.

\subsection{}
Let $f\in
\cM_{-m/2}(\rho_L)$.
For $N> 1$ and $s\in \C$, we put
\begin{equation}
\label{eq:Phi-N-s}
 \Phi_N(Z,f,s)=
\int_{F_N}\form{f(\tau)}{\Theta_L(\tau,Z)}\,
\Im(\tau)^{-1-s}d\Re(\tau)d\Im(\tau)
\qquad (Z\in \cD),
\end{equation}
where $F_N=\{\tau\in \grH\mid |\Re(\tau)|\le 1/2,1\le |\tau| \le N\}$.
Borcherds (\cite{Bo2}) showed that the limit 
\begin{equation}
\label{eq:Phi-s}
 \Phi(Z,f,s)=\lim_{N\to\infty}\Phi_N(Z,f,s)
\end{equation}
exists
if $\Re(s)$ is sufficiently large, and $\Phi(Z,f,s)$ is continued to
a meromorphic function of $s$ in $\C$.
Denote by $\Phi(Z,f)$ the constant term of the Laurent expansion  of
$\Phi(Z,f,s)$ at $s=0$.
The following fundamental 
result is due to Borcherds (\cite{Bo2}, Theorem 13.3; see
also  \cite{Br2}, \textbf{3.4}).

\begin{theorem}
\label{th:Borcherds}
 Let $f\in \cM_{-m/2}(\rho_L)$ and suppose that
\begin{equation}
\label{eq:input-data}
 c(\alpha,n)\in \Z\qquad \text{for }\alpha\in L^*/L \text{ and }n\le 0.
\end{equation}
Then there exists a meromorphic  automorphic form $\Psi_f(Z)$ on 
$\Gamma^*(L)$ of weight $c(0,0)/2$ and some multiplier system of
 $\Gamma^*(L)$ of finite order
satisfying the following properties.
\begin{enumerate}
 \item The divisor of  $\Psi_f(Z)$ is given by
\[
 \dfrac{1}{2} \sum_{\alpha\in L^*/L}
\sum_{n\in q(\alpha)+\Z,\,n<0}
  c(\alpha,n)H(\alpha,n).
\]
Here  the multiplicities  of $H(\alpha,n)$ are $2$ (respectively $1$)
if $2\alpha=0$ (respectively if $2\alpha\ne 0$) in $L^*/L$.

 \item We have
\[
 \log |\Psi_f(Z)|=-\dfrac{1}{4}\Phi(Z,f)-\dfrac{c(0,0)}{2}
\left(\log Q_1[\Im(Z)]+\Gamma'(1)/2+\log \sqrt{2\pi}\right).
\]
 \item For each Weyl chamber $W$ with respect to $f$, 
the product expansion
\[
 \Psi_f(Z)=C\,\e(Q_1(Z,\rho_f(W)))\times
\prod_{\lambda\in L_1^*,(\lambda,W)>0}\left(1-\e(Q_1(\lambda,Z))\right)
^{c(\widehat \lambda,Q_1[\lambda]/2)}
\]
holds if $Q_1[\Im(Z)]\gg 0$ and $\Im(Z)/\sqrt{Q_1[\Im(Z)]}\in W$.
Here $C$ is a constant of absolute value $1$,
$\rho_f(W)\in V_1$ is the Weyl vector attached to $(f,W)$
and $\widehat \lambda$ is the element of $L^*/L$ whose natural
projection to $L_1^*/L_1$ is $\lambda+L_1$.
(For the definitions of Weyl chambers and Weyl vectors, see \cite{Bo2}.)
\end{enumerate}

\end{theorem}

We also have the following converse theorem due to Bruinier
(\cite{Br2}, Theorem 5.12).

\begin{theorem}
\label{th:converse-th}
 Let $\Psi$ be a meromorphic automorphic form on $\Gamma^*(L)$
whose divisor is a $\Z$-linear combination of
 Heegner divisors.
Then there exists a weakly holomorohic modular form $f\in
 \cM_{-m/2}(\rho_L)$ satisfying (\ref{eq:input-data})
such that $\Psi$ is a nonzero constant multiple of 
$\Psi_f$.
\end{theorem}

We call $\Psi$ satisfying the condition of Theorem
\ref{th:converse-th} a \emph{Borcherds lift} on $\Gamma^*(L)$.

\subsection{Multiplicative symmetry for Borcherds lifts}

\begin{theorem}
\label{th:main-2}
 Let $\Psi$ be a Borcherds lift on $\Gamma^*(L)$.
Then $\Psi$ satisfies the multiplicative symmetry.
\end{theorem}

\begin{proof}
 We may suppose that  $\Psi=\Psi_f$ 
with some $f\in \cM_{-m/2}(\rho_L)$.
By Theorem \ref{th:main-1} and the definition of
$\Phi(Z,f)$, $Z\mapsto \Phi(Z,f)$ satisfies the additive symmetry.
The multiplicative symmetry for $\Psi_f$ follows directly 
 from
this fact in view of Theorem \ref{th:Borcherds} (ii).
\end{proof}

\section{Automorphic Green functions}

\subsection{Automorphic Green functions}

The automorphic Green functions associated with
Heegner divisors have been introduced by Bruinier (\cite{Br1},
\cite{Br2})
and Oda and Tsuzuki (\cite{OT}) independently.
Note that the unitary group case is also studied by Oda and Tsuzuki.
The automorphic Green functions are also studied by Bruinier and K\"{u}hn
(\cite{BK}) from an arithmetic point of view.
In this section, we recall the definition of the automorphic Green
functions mainly after \cite{BK}, and show their additive symmetry
 assuming Theorem \ref{th:main-1}.

Put
$\kappa=(m+4)/2$.
Let $\alpha\in L^*/L$ and $n\in q(\alpha)+\Z$ with $n<0$.
For $Z\in\cD\setminus H(\alpha,n)$ and $s\in\C$ with
$\Re(s)>\kappa/2$, we set
\begin{align}
\label{eq:AGF}
 \Phi_{\alpha,n}(Z,s)&=2\,\dfrac{\Gamma(s+\kappa/2-1)}{\Gamma(2s)}\\
&\qquad\times
\sum_{\lambda\in\alpha+L,\,q(\lambda)=n}
\left(\dfrac{n}{n-q(\lambda_Z)}\right)^{s+\kappa/2-1}
F\left(s+\dfrac{\kappa}{2}-1,s-\dfrac{\kappa}{2}+1,2s\, ; \,
\dfrac{n}{n-q(\lambda_Z)}
\right), \nonumber
\end{align}
where $\lambda_Z=\lambda_{v_Z}$ and
\begin{align*}
 F(a,b,c;z)&=\sum_{k=0}^{\infty}\dfrac{(a)_k(b)_k}{(c)_k}\dfrac{z^k}{k!},\\
(a)_k&=\dfrac{\Gamma(a+k)}{\Gamma(a)}.
\end{align*}
Note that
$q(\lambda_Z)=(2Q_1[\Im(Z)])\inv|Q(\lambda,\widetilde
Z)|^2$.
The series (\ref{eq:AGF})
converges locally uniformly for $Z\in\cD\setminus
H(\alpha,n)$
and $\Re(s)>\kappa/2$.
It is easy to see that $\Phi_{\alpha,n}(Z,s)$ is $\Gamma^*(L)$-invariant
in $Z$.

\begin{theorem}[\cite{Br1}, \cite{Br2}, \cite{OT}]
 \begin{enumerate}
  \item The function $\Phi_{\alpha,n}(Z,s)$
has a meromorphic continuation in $s$ to a neighborhood of $\kappa/2$
with a simple pole at $s=\kappa/2$.
  \item As a function of $Z$, $\Phi_{\alpha,n}(Z,s)$
is real analytic on $\cD\setminus H(\alpha,n)$
and has a logarithmic singularity along $H(\alpha,n)$.
Namely, if $U$ is a compact neighborhood of any $Z_0\in\cD$,
there exists a finite set $S(U)$
of $\lambda\in\alpha+L$ with $q(\lambda)=n$ such that
\[
 \Phi_{\alpha,n}(Z,s)=-4\sum_{\lambda\in S(U)}\log q(\lambda_Z)+O(1)
\]
on $U$.
  \item Let $\Delta$ denote the $G(\R)$-invariant Laplace
operator on $\cD$ induced by the Casimir element of the Lie algebra
of $G(\R)$ normalized as in \cite{Br2}, page 72.
Then we have
\[
 \Delta \Phi_{\alpha,n}(Z,s)=\Lambda(s)\Phi_{\alpha,n}(Z,s),
\]
where
\[
 \Lambda(s)=\half \left(s-\dfrac{\kappa}{2}\right)
\left(s+\dfrac{\kappa}{2}-1\right).
\]
  \item Suppose that $\Re(s)>\kappa/2$.
Then $\Phi_{\alpha,n}(Z,s)\in L^1(\cX_L)$.
If $f$ is a smooth bounded function on $\cX_L$ with
$\Delta f=\Lambda_f f$, then
\[
 \int_{\cX_L}\Phi_{\alpha,n}(Z,s)\,f(Z)\,\Omega^{m+2}=
-\dfrac{m+2}{2\Gamma(s-\kappa/2+1)}\dfrac{1}{\Lambda_f-\Lambda(s)}
\int_{H(\alpha,n)}f(Z)\Omega^{m+1}.
\]
Here $\Omega=-dd^c \log Q(\widetilde Z,\ol{\widetilde Z})$.
 \end{enumerate}
\end{theorem}

We call $\Phi_{\alpha,n}(Z,s)$ the \emph{automorphic Green function}
associated with the Heegner divisor $H(\alpha,n)$.

\subsection{The additive symmetry for automorphic Green functions}

\begin{theorem}
\label{th:main-3}
 For $\alpha\in
L^*/L$ and $n\in q(\alpha)+\Z$ with $n<0$,
the automorphic Green function $\Phi_{\alpha,n}(Z,s)$
satisfies the additive symmetry in $Z$.
\end{theorem}

\begin{proof}
 To prove the theorem, we  recall the result of Bruinier (\cite{Br2}).
Put $k=-m/2$ and 
\begin{align*}
 \bM_s(y)&=y^{s-k/2}e^{-y/2}\sum_{l=0}^{\infty}\dfrac{(s+k/2)_l}{(2s)_l\,l!}\,
y^l\qquad (y>0, s\in\C).
\end{align*}
Define the non-holomorphic Poincar\'{e} series
\[
 F_{\alpha,n}(\tau,s)=
\dfrac{1}{2\Gamma(2s)}\,
\sum_{(M,\phi)\in\widetilde \Gamma\oo\bs \Mp_2(\Z)}\left[
\bM_s(4\pi |n|y)\gre_{\alpha}(nx)
\right]|_{k}(M,\phi),
\]
where $\tau=x+iy\in\grH$, $s\in\C$ with $\Re(s)>1$ and $\widetilde
\Gamma\oo$
is the subgroup of $\Mp_2(\Z)$ generated by $T$.
Then we have
\[
 \Phi_{\alpha,n}(Z,s)=\lim_{u\to\infty}
\int_{\cF_u}\form{F_{\alpha,n}(\tau,s)}
{\Theta_{L}(\tau,Z)}
\,y\dfrac{dxdy}{y^2},
\]
where $\cF_u=\{\tau=x+iy\mid |\tau|\ge 1, |x|\ge 1/2, y\le u\}$
 if $\Re(s)$ is sufficiently
large 
(\cite{Br2}, \textbf{2.2} and \textbf{2.3}).
The additive symmetry for $\Phi_{\alpha,n}(Z,s)$ 
follows  directly from that of $\Theta_{L}(\tau,Z)$.
\end{proof}


\section{Proof of Theorem \ref{th:main-1}}
\subsection{Reduction}
First note that, to prove Theorem \ref{th:main-1}, it suffices to show 
\begin{align}
\label{eq:reduced}
 &\Theta_{\alpha}\left(\tau,\left(pz,\sqrt{p}w,z'\right)\right)
+\sum_{a=0}^{p-1}\Theta_{\alpha}\left(\tau,\left(p\inv (z+a),\sqrt{p}\inv
				   w,z'\right)\right)\\
&\quad \nonumber
=\Theta_{\alpha}\left(\tau,\left(z,\sqrt{p}w,pz'\right)\right)
+\sum_{a=0}^{p-1}\Theta_{\alpha}\left(\tau,\left(z,\sqrt{p}\inv w,p\inv (z'+a)\right)\right)
\end{align}
for any $\alpha\in L^*/L$ and any prime $p$.
Thrroughout this section, we fix an $\alpha\in L^*/L$ and 
a prime $p$.
Take an $\alpha_0\in L_0^*$ such that
\[
 \alpha\equiv \vecm{0\\0\\ \alpha_0\\0\\0}\pmod{L}
\]
and put
\[
 \alpha_1=(0,\alpha_0,0)\in L_1^*.
\]
Let $\cY=\{Y=\tp(y_1,\ldots,y_{m+2})\in V_1\mid Q_1[Y]>0, y_1>0\}$.
Note that $\Im(Z)\in \cY$ for $Z\in \cD$, and that
$\cY$ is naturally identified with the Grassmannian of $L_1$.
For $\tau\in \grH,Y\in \cY, r,t\in V_1$, we define the
generalized
Siegel theta function
$\theta_{\alpha_1}$ for $L_1$ by
\begin{equation}
\label{eq:gstf}
 \theta_{\alpha_1}(\tau,Y;r,t)
=
\sum_{\lambda\in\alpha_1+L_1}
\e\left(
i\,\Im(\tau)\dfrac{Q_1(\lambda+t,Y)^2}{Q_1[Y]}+
\dfrac{\ol{\tau}}{2}\,Q_1[\lambda+t]
-Q_1\left(\lambda+\dfrac{t}{2},r\right)
\right).
\end{equation}

The following formula due to Borcherds
 ([Bo2], Theorem 5.2; see also [Br], Theorem 2.4)
plays a crucial role in the
proof of Theorem \ref{th:main-1}.

\begin{proposition}
\label{prop:BF}
 We have
\[
 \Theta_{\alpha}(\tau,Z)=
\sqrt{\dfrac{Q_1[Y_Z]}{2\,\Im(\tau)}}
\sum_{c,d\in \Z}\e\left(
\dfrac{i\,|c\tau+d|^2\,Q_1[Y_Z]}{4\,\Im(\tau)}
\right)
\theta_{\alpha_1}(\tau,Y_Z;d X_Z,-c X_Z),
\]
where $X_Z=\Re(Z)$ and $Y_Z=\Im(Z)$, and $\alpha_1$ as defined above.
\end{proposition}

From now on we fix $\tau\in\grH$ and $Z\in \cD$.
Let 
\[
X= X_Z=(x,X_0,x'), Y=Y_Z=(y,Y_0,y')
\qquad (x,x'\in \R, y,y'\in \R_{>0}, X_0,Y_0\in V_0=\R^m)
\]
and  put
\begin{align*}
 &X\up_+=(px,\sqrt{p}X_0,x'), \quad Y\up_+=(py,\sqrt{p}Y_0,y'),\\
 & X\up_-(a)=(p\inv(x+a),\sqrt{p}\inv X_0,x'), \quad 
  Y\up_-=(p\inv y,\sqrt{p}\inv Y_0,y'),\\
 &X\down_+=(x,\sqrt{p}X_0,px'),  \quad Y\down_+=(y,\sqrt{p}Y_0,py'),\\
 & X\down_-(a)=(x,\sqrt{p}\inv X_0,p\inv(x'+a)), \quad 
  Y\down_-=(y,\sqrt{p}\inv Y_0,p\inv y')
\end{align*}
for $a\in \Z$.
We set
\begin{align*}
 I\up_+(c,d)&=\theta_{\alpha_1}(\tau,Y\up_+;d X\up_+,-c X\up_+),\\
 I\up_-(c,d)&=\sum_{a=0}^{p-1}\theta_{\alpha_1}(\tau,Y\up_-;d
 X\up_-(a),-c X\up_-(a)),\\
 I\down_+(c,d)&=\theta_{\alpha_1}(\tau,Y\down_+;d X\down_+,-c X\down_+),\\
 I\down_-(c,d)&=\sum_{a=0}^{p-1}\theta_{\alpha_1}(\tau,Y\down_-;d
 X\down_-(a),-c X\down_-(a))
\end{align*}
for $c,d\in \Z$. 
Observe that $Q_1[Y\up_+]=Q_1[Y\down_+]=pQ_1[Y]$
and $Q_1[Y\up_-]=Q_1[Y\down_-]=p\inv Q_1[Y]$.
In view of Proposition \ref{prop:BF}, the proof of (\ref{eq:reduced})
 is reduced to  the
following equality:

\begin{align}
\label{eq:reduced-1} 
& \sum_{c,d\in \Z}\e\left(
    p\dfrac{i\,|c\tau+d|^2\,Q_1[Y_Z]}{4\,\Im(\tau)}
       \right)
      I\up_+(c,d)
 +
p\inv 
   \sum_{c,d\in \Z}\e\left(
    p\inv\dfrac{i\,|c\tau+d|^2\,Q_1[Y_Z]}{4\,\Im(\tau)}
       \right)
      I\up_-(c,d)\\
&=\quad 
\sum_{c,d\in \Z}\e\left(
    p\dfrac{i\,|c\tau+d|^2\,Q_1[Y_Z]}{4\,\Im(\tau)}
       \right)
      I\down_+(c,d)+
p\inv 
   \sum_{c,d\in \Z}\e\left(
    p\inv\dfrac{i\,|c\tau+d|^2\,Q_1[Y_Z]}{4\,\Im(\tau)}
       \right)
      I\down_-(c,d)\nonumber
\end{align}

\begin{theorem}
\label{th:spin}
 \begin{enumerate}
  \item If $p\nmid \,m$ or $p\nmid \, n$, we have
\begin{equation}
\label{eq:spin-1}
 I\up_-(c,d)=I\down_-(c,d).
\end{equation}
  \item If $p|m$ and $p|n$, we have
\begin{equation}
\label{eq:spin-2}
 I\up_-(c,d)=p\,I\down_+(p\inv c,p\inv d)
\end{equation}
and
\begin{equation}
\label{eq:spin-3}
 I\down_-(c,d)=p\,I\up_+(p\inv c,p\inv d).
\end{equation}
 \end{enumerate}
\end{theorem}

\bigskip

We now show that  Theorem \ref{th:spin} implies
the equality (\ref{eq:reduced-1}) and hence Theorem \ref{th:main-1}.
By (\ref{eq:spin-1}), the left-hand side of (\ref{eq:reduced-1})
is equal to
\begin{align*}
&\sum_{c,d\in \Z}\e\left(
    p\dfrac{i\,|c\tau+d|^2\,Q_1[Y]}{4\,\Im(\tau)}
       \right)
      I\up_+(c,d)  +
   \sum_{c,d\in \Z}\e\left(
   p \dfrac{i\,|c\tau+d|^2\,Q_1[Y]}{4\,\Im(\tau)}
       \right)
      I\down_+(c,d)\\
&\quad  +
p\inv \sum_{(c,d)\in \Lambda_p}\e\left(
    p\inv\dfrac{i\,|c\tau+d|^2\,Q_1[Y]}{4\,\Im(\tau)}
       \right)
      I\up_-(c,d),
   \end{align*}
where $\Lambda_p=\{(c,d)\in \Z^2\mid \text{$p\nmid m$ or $p \nmid
n$}\}$.
Similarly, by (\ref{eq:spin-3}),
the right-hand side of (\ref{eq:reduced-1})
is equal to 
\begin{align*}
& \sum_{c,d\in \Z}\e\left(
    p\dfrac{i\,|c\tau+d|^2\,Q_1[Y]}{4\,\Im(\tau)}
       \right)
      I\down_+(c,d)  +
   \sum_{c,d\in \Z}\e\left(
    p\dfrac{i\,|c\tau+d|^2\,Q_1[Y]}{4\,\Im(\tau)}
       \right)
      I\up_+(c,d)\\
&\quad +
p\inv \sum_{(c,d)\in \Lambda_p}\e\left(
    p\inv\dfrac{i\,|c\tau+d|^2\,Q_1[Y]}{4\,\Im(\tau)}
       \right)
      I\down_-(c,d).
   \end{align*}
The equality (\ref{eq:reduced-1}) now follows from (\ref{eq:spin-1}).

\subsection{Generalized Siegel theta functions}

To prove Theorem \ref{th:spin}, we need the following formulas for
generalized Siegel theta functions.

\begin{lemma}
\label{lem:gstf-modified}
Let $X=(x,X_0,x')\in V_1$ and $Y=(y,Y_0,y')\in \cY$.
\begin{enumerate}
 \item We have
\begin{align}
\label{eq:gstf-1}
 \lefteqn{\theta_{\alpha_1}(\tau,Y; dX,-cX)}\\
 &=\e\left(\dfrac{cd}{2}(2xx'-S[X_0])\right)
\notag \\
&\times
\sum_{\substack{l,\,l'\in\Z \\ \lambda_0\in \alpha_0+L_0}}
\e\left(
\dfrac{i\,\Im(\tau)}{Q_1[Y]}\Bigl\{
(l-cx)y'+(l'-cx')y-S(\lambda_0-cX_0,Y_0)
\Bigr\}^2 \right. \notag\\
&\left. \qquad+\dfrac{\ol{\tau}}{2}
\Bigl(
2(l-cx)(l'-cx')-S[\lambda_0-cX_0]
\Bigr)
-d\left(lx'+l'x-S(\lambda_0,X_0)\right)
\right). \notag
\end{align}
 \item We have
\begin{align}
\label{eq:gstf-2}
&\theta_{\alpha_1}(\tau,Y;dX,-cX)\\
&=\sqrt{\dfrac{Q_1[Y]}{2y^2\Im(\tau)}}\ 
\e\left(
\dfrac{cd}{2}(2xx'-S[X_0])\right)
\notag\\
&
\times \sum_{\substack{l\in \Z\\ \lambda_0\in \alpha_0+L_0}}
\e\left(
\dfrac{i\,\Im(\tau)}{Q_1[Y]}
\Bigl\{
(l-cx)y'-cx'y-S(\lambda_0-cX_0,Y_0)
\Bigr\}^2 \right.\notag\\
&\left.\qquad \qquad \qquad
+\dfrac{\ol{\tau}}{2}
\Bigl(
-2cx'(l-cx)-S[\lambda_0-cX_0]
\Bigr)
-d\left(lx'-S(\lambda_0,X_0)\right)
\right)
\notag\\
& \sum_{l'\in \Z}
\e\left(
\dfrac{iQ_1[Y]}{4y^2\Im(\tau)}
\Bigl\{
l'+\dfrac{2iy\Im(\tau)}{Q_1[Y]}\left(
(l-cx)y'-cx'y-S(\lambda_0-cX_0,Y_0)
\right)
\right.\notag\\
&\qquad\qquad\qquad
+\ol{\tau}(l-cx)-dx
\left.
\Bigr\}^2
\right)\notag
\end{align}
and
\begin{align}
\label{eq:gstf-3}
& \theta_{\alpha_1}(\tau,Y;dX,-cX) \\
&=\sqrt{\dfrac{Q_1[Y]}{2(y')^2\Im(\tau)}}\ 
\e\left(
\dfrac{cd}{2}(2xx'-S[X_0])\right)
\notag\\
&\quad
\times \sum_{\substack{l'\in \Z\\ \lambda_0\in \alpha_0+L_0}}
\e\left(
\dfrac{i\,\Im(\tau)}{Q_1[Y]}
\Bigl\{
-cxy'+(l'-cx')y-S(\lambda_0-cX_0,Y_0)
\Bigr\}^2 \right.\notag\\
&\left.\qquad \qquad \qquad
+\dfrac{\ol{\tau}}{2}
\Bigl(
-2cx(l'-cx')-S[\lambda_0-cX_0]
\Bigr)
-d\left(l'x-S(\lambda_0,X_0)\right)
\right)
\notag\\
&\qquad \sum_{l\in \Z}
\e\left(
\dfrac{iQ_1[Y]}{4(y')^2\Im(\tau)}
\Bigl\{
l+\dfrac{2iy'\Im(\tau)}{Q_1[Y]}\left(
-cxy'+(l'-cx')y-S(\lambda_0-cX_0,Y_0)
\right)
\right.\notag\\
&\qquad\qquad\qquad
+\ol{\tau}(l'-cx')-dx'
\left.
\Bigr\}^2
\right)\notag.
\end{align}
\end{enumerate}
\end{lemma}

\begin{proof}
The first assertion is immediate from  (\ref{eq:gstf}).
The equality (\ref{eq:gstf-2})
(respectively (\ref{eq:gstf-3}))
is obtained by applying the Poisson summation formula
to the sum over $l'\in\Z$ (respectively $l\in\Z$)
in the right-hand side of (\ref{eq:gstf-1}).
\end{proof}

\subsection{Proof of Theorem \ref{th:spin} (i)}
In this and the next subsections, we keep the notation of 
 \textbf{7.1}.
In view of (\ref{eq:gstf-1}), we have
\begin{align}
\label{eq:up-minus-1}
&I\up_-(c,d) \\
\notag &=\sum_{\substack{l,\,l'\in \Z \\ \lambda_0\in \alpha_0+ L_0}}
 \sum_{a=0}^{p-1}\e\left(
\dfrac{cd}{2}\,\Bigl(
2p\inv (x+a)x'-p\inv S[X_0]
\Bigr)
\right.\\
&\qquad\qquad\notag
\left.
+p\,\dfrac{i\,\Im(\tau)}{Q_1[Y]}
\Bigl\{
\left(l-p\inv c(x+a)\right)y'
+p\inv (l'-cx')y
-S(\lambda_0-c\sqrt{p}\inv X_0,\sqrt{p}\inv Y_0)
\Bigr\}^2
\right.\\
&\notag\left.
\qquad \qquad +\dfrac{\ol{\tau}}{2}\Bigl(
2\left(l-p\inv c(x+a)\right)(l'-cx')-S[\lambda_0-c\sqrt{p}\inv X_0]
\Bigr)
\right.\\
&\notag \left.\qquad\qquad 
-d\bigl(
lx'+p\inv l'(x+a)-S(\lambda_0,\sqrt{p}\inv X_0)
\bigr)
\right)
\end{align}
and
\begin{align}
\label{eq:down-minus-1}
& I\down_-(c,d) \\
\notag&=\sum_{\substack{l,\,l'\in \Z \\ \lambda_0\in \alpha_0+ L_0}}
 \sum_{a=0}^{p-1}\e\left(
\dfrac{cd}{2}\,\Bigl(
2p\inv x(x'+a)-p\inv S[X_0]
\Bigr)
\right.\\
&\qquad\qquad\notag
\left.
+p\,\dfrac{i\,\Im(\tau)}{Q_1[Y]}
\Bigl\{
p\inv (l-cx)y'
+\left(l'-p\inv c(x'+a)\right)y
-S(\lambda_0-c \sqrt{p}\inv X_0,\sqrt{p}\inv Y_0)
\Bigr\}^2
\right.\\
&\notag\left.
\qquad \qquad +\dfrac{\ol{\tau}}{2}\Bigl(
2(l-cx)\left(l'-p\inv c(x'+a)\right)-S[\lambda_0-c \sqrt{p}\inv X_0]
\Bigr)
\right.\\
&\notag \left.\qquad\qquad 
-d \bigl(
p\inv l(x'+a)+l'x-S(\lambda_0,\sqrt{p}\inv X_0)
\bigr)
\right).
\end{align}
First suppose that $p|c$ and $p\nmid \,d$.
Changing $l$ into $l+p\inv ca$ in the sum (\ref{eq:up-minus-1}), we obtain
\begin{align*}
& I\up_-(c,d)\\
&=
\e\left(
\dfrac{cd}{2p}(2xx'-S[X_0])
\right)\\
&\quad \sum_{\substack{l,l'\in \Z\\ \lambda_0\in \alpha_0+L_0}}
\e\left(
p\dfrac{i\,\Im(\tau)}{Q_1[Y]}\Bigl\{
(l-p\inv cx)y'+p\inv (l'-cx')y-S(\lambda_0-c\sqrt{p}\inv
 X_0,\sqrt{p}\inv Y_0)
\Bigr\}^2 \right. \\
&\left. \qquad +\dfrac{\ol{\tau}}{2}
\Bigl(
2(l-p\inv cx)(l'-cx')-S[\lambda_0-c \sqrt{p}\inv X_0]
\Bigr)
-d \bigl(lx'+p\inv l'x-S(\lambda_0,\sqrt{p}\inv X_0)
\bigr)
\right)\\
& \quad \times\sum_{a=0}^{p-1}\e\left(
-\dfrac{dl'}{p}a
\right).
\end{align*}
Since the last sum is equal to $p\, \delta(p|l')$, we have
\begin{align*}
& I\up_-(c,d)\\
\notag &=p\ 
\e\left(
\dfrac{cd}{2p}(2xx'-S[X_0])
\right)\\
&\quad \notag
\sum_{\substack{l,l'\in \Z\\ \lambda_0\in \alpha_0+L_0}}
\e\left(
p\dfrac{i\,\Im(\tau)}{Q_1[Y]}\Bigl\{
(l-p\inv cx)y'+ (l'-p\inv cx')y-S(\lambda_0-c \sqrt{p}\inv
 X_0,\sqrt{p}\inv Y_0)
\Bigr\}^2 \right. \\
&\left. \notag \qquad +\dfrac{\ol{\tau}}{2}
\Bigl(
2p(l-p\inv cx)(l'-p\inv cx')-S[\lambda_0-c \sqrt{p}\inv X_0]
\Bigr)
-d \bigl(lx'+l'x-S(\lambda_0,\sqrt{p}\inv X_0)
\bigr)
\right).
\end{align*}
A similar calculation shows  that $I\down_-(c,d)$ is equal to
 the right-hand side
of the above equality. Thus the equality (\ref{eq:spin-1}) has been proved
in the case $p|c$ and $p\nmid \, d$.

Next suppose that $p\nmid \, c$.
Then $l_1=pl-ca$ runs over $\Z$ as $l$ runs over $\Z$ and 
$a$ over $\{0,1,\ldots,p-1\}$.
Take an integer $c_0$ such that $cc_0\equiv 1\pmod{p}$.
Then $a\equiv -c_0l_1\pmod{p}$.
In view of  (\ref{eq:up-minus-1}), we have
\begin{align*}
& I\up_-(c,d)\\
&=
\e\left(
 \dfrac{cd}{2p}(2xx'-S[X_0])
\right)\\
&\quad 
\sum_{\substack{l_1,l'\in \Z\\ \lambda_0\in \alpha_0+L_0}}
\e\left(
p\dfrac{i\,\Im(\tau)}{Q_1[Y]}\Bigl\{
p\inv (l_1-cx)y'+ p\inv (l'-cx')y-S(\lambda_0-c\sqrt{p}\inv
 X_0,\sqrt{p}\inv Y_0)
\Bigr\}^2 \right. \\
&  \qquad +\dfrac{\ol{\tau}}{2}
\Bigl(
2p\inv (l_1-cx)(l'-cx')-S[\lambda_0-c\sqrt{p}\inv X_0]
\Bigr)
-d \bigl(p\inv l_1x'+p\inv l'x-S(\lambda_0,\sqrt{p}\inv X_0)
\bigr) \\
&\left. \qquad 
+p\inv c_0 d l_1l'
\right).
\end{align*}
On the other hand, since 
$l'_1=pl'-ca$ runs over $\Z$ as $l$ runs over $\Z$ and
$a$ over $\{0,1,\ldots,p-1\}$, and since
$a\equiv -c_0l_1'\pmod{p}$, we obtain
\begin{align*}
& I\down_-(c,d)\\
&=
\e\left(
\dfrac{cd}{2p}(2xx'-S[X_0])
\right)\\
&\quad 
\sum_{\substack{l,l_1'\in \Z\\ \lambda_0\in \alpha_0+L_0}}
\e\left(
p\dfrac{i\,\Im(\tau)}{Q_1[Y]}\Bigl\{
p\inv (l-cx)y'+ p\inv (l_1'-cx')y-S(\lambda_0-c\sqrt{p}\inv
 X_0,\sqrt{p}\inv Y_0)
\Bigr\}^2 \right. \\
&  \qquad +\dfrac{\ol{\tau}}{2}
\Bigl(
2p\inv (l-cx)(l_1'-cx')-S[\lambda_0-c\sqrt{p}\inv X_0]
\Bigr)
-d \bigl(p\inv lx'+p\inv l_1'x-S(\lambda_0,\sqrt{p}\inv X_0)
\bigr) \\
&\left. \qquad 
+p\inv c_0 d ll_1'
\right).
\end{align*}
Comparing these two expressions, we obtain
 the equality (\ref{eq:spin-1}) in the case $p\nmid \, c$.

\subsection{Proof of Theorem \ref{th:spin} (ii)}

Suppose that $p|c$ and $p|d$.
By (\ref{eq:gstf-2}), we have
\begin{align*}
& I\up_-(c,d)\\
&=\sqrt{p}\sqrt{\dfrac{Q_1[Y]}{2y^2 \Im(\tau)}}
  \sum_{a=0}^{p-1}\e\left(
\dfrac{cd}{2}\Bigl(
2p\inv (x+a)x'-p\inv S[X_0]
\Bigr)
\right)\\
&\quad
\sum_{\substack{l\in\Z\\ \lambda_0\in \alpha_0+L_0}}
\e\left(
p\,\dfrac{i\,\Im(\tau)}{Q_1[Y]}\Bigl\{
\left(l-p\inv c (x+a)\right)y'-p\inv cx'y-S(\lambda_0-c\sqrt{p}\inv
 X_0,\sqrt{p}\inv Y_0)
\Bigr\}^2
\right. \\
& \left.\qquad\qquad
+\dfrac{\ol{\tau}}{2}\Bigl(
-2cx'\left(l-p\inv c(x+a)\right)-S[\lambda_0-c \sqrt{p}\inv X_0]
\Bigr)
-d\left(lx'-S(\lambda_0,\sqrt{p}\inv X_0)\right)
\right)\\
&\quad \sum_{l'\in \Z}
\e\left(
p\,\dfrac{iQ_1[Y]}{4y^2\Im(\tau)}\Bigl\{
l'+\dfrac{2iy\Im(\tau)}{Q_1[Y]}\Bigl(
\left(l-p\inv c(x+a)\right)y'
-p\inv cx'y-S(\lambda_0-c\sqrt{p}\inv X_0,\sqrt{p}\inv Y_0)
\Bigr)\right.\\
&\left.\qquad\qquad
+\ol{\tau}\left(l-p\inv c(x+a)\right)-p\inv d(x+a)
\Bigr\}^2
\right).
\end{align*}
Changing $l$ into $l+p\inv ca$ and $l'$ into $l'+p\inv da$
respectively, we see that
$I\up_-(c,d)$ is equal to
\begin{align}
\label{eq:Iupminus}
&p\sqrt{p}\sqrt{\dfrac{Q_1[Y]}{2y^2 \Im(\tau)}}
  \,\e\left(
\dfrac{cd}{2p}\bigl(
2xx'- S[X_0]
\bigr)
\right)\\
\nonumber &\quad \times
\sum_{\substack{l\in\Z\\ \lambda_0\in \alpha_0+L_0}}
\e\left(
p\,\dfrac{i\,\Im(\tau)}{Q_1[Y]}\Bigl\{
(l-p\inv cx)y'
-p\inv cx'y
-S(\lambda_0-c\sqrt{p}\inv
 X_0,\sqrt{p}\inv  Y_0)
\Bigr\}^2
\right. \\
\nonumber & \left.\qquad\qquad
+\dfrac{\ol{\tau}}{2}\Bigl(
-2cx'(l-p\inv cx)-S[\lambda_0-c\sqrt{p}\inv X_0]
\Bigr)
-d\left(lx'-S(\lambda_0, \sqrt{p}\inv X_0)\right)
\right)\\
\nonumber &\quad \sum_{l'\in \Z}
\e\left(
p\,\dfrac{iQ_1[Y]}{4y^2\Im(\tau)}\Bigl\{
l'+p\inv \, \dfrac{2iy\Im(\tau)}{Q_1[Y]}\Bigl(
p(l-p\inv cx)y'-cx'y
-S(\lambda_0-c\sqrt{p}\inv X_0,\sqrt{p} Y_0)
\Bigr)\right.\\
\nonumber &\left.\qquad\qquad
+\ol{\tau}(l-p\inv cx)-p\inv dx
\Bigr\}^2
\right).
\end{align}
Using again (\ref{eq:gstf-2}), we see that $pI\down_+(p\inv c,p\inv d)$
is equal to 
\begin{align}
\label{eq:Idownplus}
 &p\sqrt{p}\,\sqrt{\dfrac{Q_1[Y]}{2y^2 \Im(\tau)}}
\,\e\left(
\dfrac{cd}{2p}\bigl(
2xx'- S[X_0]
\bigr)
\right)\\
\nonumber & \quad \times
\sum_{\substack{l\in\Z\\ \lambda_0\in \alpha_0+L_0}}
\e\left(
p\inv \,\dfrac{i\,\Im(\tau)}{Q_1[Y]}\Bigl\{
\left(l-(p\inv c)x\right)py'
-(p\inv c)px'y
-S(\lambda_0-(p\inv c)
 \sqrt{p}X_0,\sqrt{p} Y_0)
\Bigr\}^2
\right. \\
\nonumber & \left.\qquad\qquad
+\dfrac{\ol{\tau}}{2}\Bigl(
-2(p\inv c)px'\left(l-(p\inv c)x\right)-S[\lambda_0-(p\inv c)\sqrt{p} X_0]
\Bigr)
-p\inv d\left(lpx'-S(\lambda_0, \sqrt{p} X_0)\right)
\right)\\
\nonumber &\quad \sum_{l'\in \Z}
\e\left(
p\,\dfrac{iQ_1[Y]}{4y^2\Im(\tau)}\Bigl\{
l'+p\inv \, \dfrac{2iy\Im(\tau)}{Q_1[Y]}\Bigl(
\left(l-(p\inv c)x\right)py'-(p\inv c)px'y \right. \\
\nonumber &\qquad \left.-S(\lambda_0-(p\inv c)\sqrt{p} X_0,\sqrt{p} Y_0)
\Bigr)
+\ol{\tau}\left(l-(p\inv c)x\right)-(p\inv d)x
\Bigr\}^2
\right).
\end{align}
Comparing (\ref{eq:Iupminus}) and (\ref{eq:Idownplus}), we obtain 
the equality (\ref{eq:spin-2}).
The equality (\ref{eq:spin-3}) is proved in a similar manner.
Then the proof of Theorem \ref{th:spin} has been completed.

\section{Examples}

\subsection{Siegel modular forms of degree two}
In this section, we consider the case where $m=1$ and $S=(2)$.
In this case, $\cD$ is isomorphic to the Siegel upper half space $\grH_2$
of degree $2$, and the space of holomorphic automorphic forms on $\Gamma^*(L)$
of weight $k$
is naturally identified with the space $M_k(\Gamma_2)$ of holomorphic Siegel
modular forms on $\Gamma_2=\mathrm{Sp}_2(\Z)$
of weight $k$. We denote by $S_k(\Gamma_2)$ the space of cusp
forms in $M_k(\Gamma_2)$.

It is known that $S_{10}(\Gamma_2)$ and $S_{12}(\Gamma_2)$ are one dimensional.
Let $\chi_{10}$ and $\chi_{12}$ be nonzero elements of $S_{10}(\Gamma_2)$ and
$S_{12}(\Gamma_2)$ respectively. Then $\chi_{10}$ is a Borcherds lift
(\cite{GN})
and hence satisfies the multiplicative symmetry.
We will show that, on the other hand, $\chi_{12}$ does not satisfy
the multiplicative symmetry and hence is not a Borcherds lift.

\subsection{The Saito-Kurokawa lifting}

To calculate the  Fourier coefficients of $\chi_{10}$ and $\chi_{12}$,
it is convenient to express them as Saito-Kurokawa lifts (see \cite{EZ}).
Let $k,m\in\NN$.
For a holomorphic function $\phi$ on $\grH\times\C$, we put
\begin{align*}
 (\phi|_{k.m}\gamma)(\tau,z)&=(c\tau+d)^{-k}\e\left(
m\dfrac{-cz^2}{c\tau+d}\right)
            \phi\left(\dfrac{a\tau+b}{c\tau+d},\dfrac{z}{c\tau+d}\right)
\quad \left(\gamma=\Nmat{a&b}{c&d}\in \SL_2(\R)\right),\\
(\phi|_m[\lambda,\mu])(\tau,z)&=\e\left(
-m(\lambda^2\tau+2\lambda
 z)\right)
              \phi(\tau,z+\lambda \tau+\mu)
\qquad(\lambda,\mu\in \R).
\end{align*}
Let $J_{k,m}$ be the space of 
holomorphic functions $\phi$ on $\grH\times \C$ satifying
\begin{align*}
 \phi|_{k,m}\gamma&=\phi \qquad (\gamma\in \SL_2(\Z)),\\
\phi|_{m}[\lambda,\mu]&=\phi \qquad (\lambda,
\mu\in\Z)
\end{align*}
and $c_{\phi}(n,r)=0$ unless $4nm\ge r^2$, where
$\phi(\tau,z)=\sum_{n,r\in \Z}c_{\phi}(n,r)\e(n\tau+rz)$
is the Fourier expansion of $\phi$.
We call $J_{k,m}$ the space of holomorphic
Jacobi forms of weight $k$ and index $m$.
Let $J_{k,1}^{\mathrm{cusp}}=\{\phi\in J_{k,m}\mid
c_{\phi}(n,r)=0 \text{ unless }4nm>r^2
\}$ be the space of
Jacobi cusp forms of weight $k$ and index $m$.
For $\phi\in J_{k,1}^{\mathrm{cusp}}$ and $m\in\Z_{>0}$,
we put
\[
 (\phi|V_m)(\tau,z)=m^{k-1}\sum_{\xi}(c\tau+d)^{-k}\e\left(
m\dfrac{-cz^2}{c\tau+d}\right)
\phi\left(\dfrac{a\tau+b}{c\tau+d},\dfrac{z}{c\tau+d}\right),
\]
where $\xi=\Nmat{a&b}{c&d}$ 
runs over $\SL_2(\Z)\backslash \M_2(\Z)$ with $\det\xi=m$.
Then $\phi|V_m\in J_{k,m}^{\mathrm{cusp}}$.
In what follows, we write $(\tau,z,\tau')$ for $\Nmat{\tau&z}{z&\tau'}\in\grH_2$.
Define
\[
 \cV(\phi)(\tau,z,\tau')=
\sum_{m=1}^{\infty}(\phi|V_m)(\tau,z)\e(m\tau')
\qquad((\tau,z,\tau')\in \grH_2).
\]
The Saito-Kurokawa lift $\cV(\phi)$ belongs to $S_k(\Gamma_2)$ and its Fourier
expansion
is given by
\[
 \cV(\phi)(\tau,z,\tau')=\sum_{n,r,m\in \Z,\,4nm>r^2,\,n>0}A(n,r,m)
\e(n\tau+rz+m\tau'),
\]
where
\[
 A(n,r,m)=\sum_{0<d|(n,r,m)}d^{k-1}c_{\phi}\left(nm/d^2,r/d\right)
\]
(see \cite{EZ} \S 3).

\subsection{The Siegel modular forms $\chi_{10}$ and $\chi_{12}$}

For $k\in\Z$ with $k\ge 4$, put
\begin{align*}
 E_k(\tau)&=\half
 \sum_{c,d\in\Z,(c,d)=1}(c\tau+d)^{-k}\qquad(\tau\in\grH),\\
 E_{k,1}(\tau,z)&=\half
 \sum_{c,d\in\Z,(c,d)=1}\sum_{\lambda\in\Z}
(c\tau+d)^{-k}\e\left(
\lambda^2\dfrac{a\tau+b}{c\tau+d}
+2\lambda\dfrac{z}{c\tau+d}-\dfrac{cz^2}{c\tau+d}\right)
\qquad((\tau,z)\in\grH\times
 \C).
\end{align*}
Then $E_k\in M_k(\SL_2(\Z))$ and $E_{k,1}\in J_{k,1}$.
Set
\begin{align*}
 \phi_{10,1}&
    =\dfrac{1}{144}\left(E_6(\tau)E_{4,1}(\tau,z)-E_4(\tau)E_{6,1}(\tau,z)\right)\\
&=\left(\zeta-2+\zeta\inv\right)q
 +\left(-2\zeta^2-16\zeta+36-16\zeta\inv-2\zeta^{-2}\right)q^2\\
&\quad
 +\left(\zeta^3+36\zeta^2+99\zeta-272+99\zeta\inv+36\zeta^{-2}+\zeta^{-3}\right)
q^3+\cdots,\\
 \phi_{12,1}&
    =\dfrac{1}{144}\left(E_4(\tau)^2E_{4,1}(\tau,z)-E_6(\tau)E_{6,1}(\tau,z)\right)\\
&=\left(\zeta+10+\zeta\inv\right)q
 +\left(10\zeta^2-88\zeta-132-88\zeta\inv+10\zeta^{-2}\right)q^2\\
&\quad
 +\left(\zeta^3-132\zeta^2+1275\zeta+736+1275\zeta\inv-132\zeta^{-2}+\zeta^{-3}\right)
q^3+\cdots,
\end{align*}
where $q=\e(\tau),\zeta=\e(z)$.
Then $\phi_{10,1}\in J_{10,1}^{\mathrm{cusp}}$
and $\phi_{12,1}\in J_{12,1}^{\mathrm{cusp}}$.
Let $\chi_{10}=\cV(\phi_{10,1})$ and $\chi_{12}
=\cV(\phi_{12,1})$ be the Saito-Kurokawa lifts of
$\phi_{10,1}$
and $\phi_{12,1}$ respectively.
It is known that $S_{10}(\Gamma_2)=\C\chi_{10}$ 
and $S_{12}(\Gamma_2)=\C\chi_{12}$,
and that $\chi_{10}$ is a Borcherds lift.

\subsection{Relations satisfied by Fourier coefficients of Borcherds
  lifts}
Let
\[
 F(\tau,z,\tau')=\sum_{n,r,m}A_F(n,r,m)\e(n\tau+rz+m\tau')\in S_k(\Gamma_2).
\]
We make a convention that $A_F(n,r,m)=0$ unless $4nm>r^2$ and $n>0$.
Put
\begin{align*}
 F\up(\tau,z,\tau')&=F|\cT\up_2(\tau,\sqrt{2}z,\tau')\\
    &=F(2\tau,2z,\tau')F(\tau/2,z,\tau')F((\tau+1)/2,z,\tau')
\end{align*}
and
\begin{align*}
 F\down(\tau,z,\tau')&=F|\cT\down_2(\tau,\sqrt{2}z,\tau')\\
    &=F(\tau,2z,2\tau')F(\tau,z,\tau'/2)F(\tau,z,(\tau'+1)/2)).
\end{align*}
Let
\begin{align*}
 F\up(\tau,z,\tau')&=\sum_{n,r,m}A_F\up(n,r,m)\e(n\tau+rz+m\tau'),\\
 F\down(\tau,z,\tau')&=\sum_{n,r,m}A_F\down(n,r,m)\e(n\tau+rz+m\tau')
\end{align*}
be the Fourier expansions of $F\up$ and $F\down$ respectively.

\begin{lemma}
\label{lem:fourier}
 Suppose that $F$ is a Borcherds lift.
Then there exists a complex number $\epsilon$ of absolute value $1$ such
 that
$A_F\up(n,r,m)=\epsilon A_F\down(n,r,m)$ for every $n,r,m$.
\end{lemma}

\begin{proof}
 This follows from Theorem \ref{th:main-2}. 
\end{proof}

A straightforward calculation shows the following:

\begin{lemma}
Let $F\in S_k(\Gamma_2)$.
 For $r\in \Z$, we have
\begin{align*}
 A_F\up(4,r,3)&=\sum_{r_1,r_2,r_3\in\Z,\,2r_1+r_2+r_3=r}
 A_F(1,r_1,1)\left\{
A_F(2,r_2,1)A_F(2,r_3,1)\right.\\
&\qquad \left.-A_F(3,r_2,1)A_F(1,r_3,1)-A_F(1,r_2,1)A_F(3,r_3,1)
\right\}
\end{align*}
and
 \begin{align*}
  A_F\down(4,r,3)&=\sum_{r_1,r_2,r_3\in\Z,\,2r_1+r_2+r_3=r}
 \left\{
-A_F(2,r_1,1)A_F(1,r_2,1)A_F(1,r_3,1)\right.\\
&\qquad \left.-A_F(1,r_1,1)A_F(2,r_2,1)A_F(1,r_3,1)
-A_F(1,r_1,1)A_F(1,r_2,1)A_F(2,r_3,1)
\right\}.
 \end{align*}
\end{lemma}

The values of $A_F\up(4,r,3)$ and $A_F\down(4,r,3)$ for several $r$ with
$F=\chi_{10}$ and $F=\chi_{12}$ are
given as follows.

\bigskip

\begin{center}
 
\begin{tabular}{|c|c|c|c|c|}
\hline
$r$ & 0 & 1 & 2 & 3 \\
\hline\hline
$A_{\chi_{10}}\up(4,r,3)$ & -552 & 216 & 222 & -212  \\
\hline
$A_{\chi_{10}}\down(4,r,3)$ & -552 & 216 & 222 & -212  \\
\hline\hline
$A_{\chi_{12}}\up(4,r,3)$ & 143304 & -59112 & 65310 & -20396  \\
\hline
$A_{\chi_{12}}\down(4,r,3)$ & 43512 & 26424 & 11850 & 3364  \\
\hline
\end{tabular}

\end{center}

This table together with Lemma \ref{lem:fourier} show the following:

\begin{theorem}
  The Siegel cusp form $\chi_{12}$ is not a Borcherds lift.
\end{theorem}


\bigskip

\noindent
Bernhard Heim\\
\noindent
German University of Technology in Oman,
Way No. 36,
Building No. 331,
North Ghubrah, Muscat,
Sultanate of Oman\\
e-mail: bernhard.heim@gutech.edu.om

\bigskip

\noindent
Atsushi Murase\\
\noindent
Department of Mathematics, Faculty of Science, 
Kyoto Sangyo University, Motoyama, Kamigamo, 
Kita-ku, Kyoto 603-8555, Japan\\
e-mail: murase@cc.kyoto-su.ac.jp


\vspace{1cm}
\begin{thebibliography}{999999}
 \bibitem[Bo1]{Bo1}
R. E. Borcherds,
\emph{Automorphic forms on $O_{s+2,2}(\R)$ and infinite products},
Invent. Math. \textbf{120} (1995), 161--213.
 \bibitem[Bo2]{Bo2}
R. E. Borcherds,
\emph{Automorphic forms with singularities on Grassmannians},
Invent. Math. \textbf{132} (1998), 491--562.
 \bibitem[Br1]{Br1}
J. H. Bruinier,
\emph{Borcherds products and Chern classes of Hirzeburch-Zagier
	      divisors},
Invent. Math. \textbf{138} (1999), 51--83. 
 \bibitem[Br2]{Br2}
J. H. Bruinier,
\emph{Borcherds Products on $O(2,l)$ and Chern Classes of Heegner
	     Divisors},
Lecture Notes in Math. \textbf{1780} (2002), Springer Verlag.
 \bibitem[BK]{BK}
J. H. Bruinier and U. K\"{u}hn,
\emph{Integrals of automorphic Green's functions associated to
Heegner divisors}, Int. Math. Res. Notices (2003) No. 31, 1687--1729.
\bibitem[EZ]{EZ}
M. Eichler and D. Zagier,
\emph{The theory of Jacobi forms} (1985), Springer Verlag.
\bibitem[GN]{GN}
V. A. Gritsenko and V. V. Nikulin,
\emph{Siegel automorphic form correction of some Lorentzian Kac-Moody
	    Lie algebra},
Amer. J. Math. \textbf{119} (1997), 181--224.

 \bibitem[GZ]{GZ}
B. Gross and D. Zagier,
\emph{Heegner points and derivatives of $L$-series},
Invent. Math. \textbf{84} (1986), no. 2, 225-320.
\bibitem[HaMo]{HaMo}
J. A. Harvey and G. Moore,
\emph{Algebras, BPS states, and strings},
Nuclear Physics B \textbf{463} (1996), 315--368.
 \bibitem[H]{H}
B. Heim, 
\emph{On the Spezialschar of Maass},
arXiv:0801.1804.
 \bibitem[HeMu]{HeMu}
B. Heim and A. Murase,
\emph{A characterization of the Maass space on $O(2,m+2)$ by symmetries},
arXiv:1003.0573.
\bibitem[HZ]{HZ}
F. Hirzeburch and D. Zagier,
\emph{Intersection numbers of curves on Hilbert modular surfaces  and
modular forms of Nebentypus},
Invent. Math. \textbf{36} (1976), 57--113.
 \bibitem[OT]{OT}
T. Oda and M. Tsuzuki,
\emph{Automorphic Green functions associated with the secondary
spherical functions},
Publ. Res. Inst. Math. Sci. \textbf{39} (2003), no. 3, 451--533.
\bibitem[Si]{Si}
C. L. Siegel,
\emph{Indefinite quadratische Formen und Funktionentheorie I.},
Math. Ann.  \textbf{124} (1951), 17--54.
\bibitem[vdG]{vdG}
G. van der Geer,
\emph{Hirbert modular surfaces},
Ergebnisse der mathematik und ihre Grenzgebiete (3), \textbf{16},
Springer-Verlag, Berlin, 1988.
\end{thebibliography}
\end{document}